\documentclass[reqno, 11pt]{amsart}
\usepackage[]{hyperref}

\usepackage{bbm}
\usepackage{url}
\usepackage{stmaryrd}
\usepackage{mathrsfs}
\usepackage{cases}
\usepackage{amsfonts}
\usepackage{amssymb}
\usepackage{amsmath}
\usepackage{enumerate}
\usepackage{tikz}
\usepackage{extarrows}

\allowdisplaybreaks[4]
\newtheorem{theorem}{Theorem}[section]
\newtheorem{lemma}[theorem]{Lemma}
\newtheorem{proposition}[theorem]{Proposition}

\theoremstyle{definition}

\newtheorem{remark}[theorem]{Remark}
\numberwithin{equation}{section}

\let\a=\alpha \let\al=\alpha
\let\b=\beta
\let\d=\delta
\let\D=\Delta

\let\f=\frac
\let\g=\gamma
\let\G=\Gamma
\let\i=\infty 

\let\l=\ell

\let\r=\rho

\let\t=\tau

\let\pa=\partial
\let\th=\theta
\let\na=\nabla
\let\wt=\widetilde
\let\wh=\widehat
\let\vph=\varphi

\def\bbR{\mathbb{R}}

\def\bbZ{\mathbb{Z}}
\def\bbN{\mathbb{N}}

\def\scrF{\mathscr{F}}

\def\md{\,\mathrm{d}}
\def\me{\mathrm{e}}
\def\mi{\mathrm{i}}

\newcommand{\be}{\begin{equation*}}
\newcommand{\ee}{\end{equation*}}
\newcommand{\ben}{\begin{equation}}
\newcommand{\een}{\end{equation}}
\newcommand{\bn}{\begin{enumerate}}
\newcommand{\en}{\end{enumerate}}

\def\lp{L^p}

\def\mpqa{M^{p,q}_{0,\al}}

\def\mpqsa{M^{p,q}_{s,\al}}

\def\mpqda{M^{p,q}_{\d n,\al}}

\def\bka{\Box_k^{\al}}
\def\bkat{\widetilde{\Box}_k^{\al}}
\def\aa{\frac{\al}{1-\al}}
\def\naa{\frac{-\al}{1-\al}}

\def\ata{\frac{2\al}{1-\al}}

\def\lkr{\langle k \rangle}             \def\llr{\langle \l \rangle}
\def\ka{\lkr^{\frac{1}{1-\al}}}         
\def\kaa{\lkr^{\frac{\al}{1-\al}}}      \def\laa{\llr^{\frac{\al}{1-\al}}}

\def\kaak{\kaa k}                       \def\laal{\laa \l}

\begin{document}
\title[Unimodular multipliers on $\alpha$-modulation spaces]
{Unimodular multipliers on $\alpha$-modulation spaces: A revisit with new method under weaker conditions}
\author{GUOPING ZHAO}
\address{School of Applied Mathematics, Xiamen University of Technology, Xiamen, 361024, P.R.China}
\email{guopingzhaomath@gmail.com}
\author{WEICHAO GUO}
\address{School of Science, Jimei University, Xiamen, 361021, P.R.China}
\email{weichaoguomath@gmail.com}
\thanks{}
\subjclass[2010]{42B15, 42B35,42B37}
\keywords{Unimodular multiplier, $\alpha$-modulation space, Fourier multiplier.}

\begin{abstract}
By a new method derived from Nicola--Primo--Tabacco \cite{NicolaPrimoTabacco2018JoPOaA},
we study the boundedness on $\a$-modulation spaces of unimodular multipliers with symbol $\me^{\mi\mu(\xi)}$.
Comparing with the previous results,
the boundedness result is established for a larger family of unimodular multipliers under weaker assumptions.
\end{abstract}

\maketitle

\section{Introduction and Preliminary}
In this paper we study the Fourier multiplier in $\bbR^n$ of the form
\[
\me^{\mi\mu(D)}f:=\scrF^{-1}\me^{\mi\mu(\xi)}\scrF f(\xi),
\]
which is called the unimodular multiplier with symbol $\me^{\mi\mu(\xi)}$, where $\mu(\xi)$ is a real-valued function,
$\scrF$ and $\scrF^{-1}$ are Fourier transform and inverse Fourier transform respectively.

In order to solve the Cauchy problem for the nonlinear Schr\"{o}dinger equation (NLS)
\[\mi u_t + \D u=f(u), u(0,x)=u_0(x),\]
where $\mi=\sqrt{-1}$, $\D=\sum\limits_{j=1}^n \partial_j^2$ and $f(u)$ is a nonlinear function,
e.g. $f(u)=|u|^{\b}u$ with $\b>0$,
one always considers its equivalent integral form
\[
  u(t)= \me^{\mi t\D} u_0  -  \mi \int_0^t \me^{\mi (t-\t)\D}f(u(\t))\md\t,
\]
and needs to make some elaborate (semi-)norm estimates for the linear part and the nonlinear part of the above integral form,
e.g. Strichartz estimate, more precisely, the boundedness of $e^{it\D}$ on function spaces.
It is known that
$\me^{\mi t\D}: L^p\rightarrow L^p$ is bounded iff $p=2$, which is one of the reasons that we can not solve NLS in $L^p$ $(p\neq2)$.
Similar situation happens to the Besov spaces, i.e. $\me^{\mi t\D}$ is bounded on $B^{p,q}_s$ iff $p=2$, see \cite{Mizuhara1987MN}.
It is well-known that the (inhomogenous) Besov space $B^{p,q}_s$ is a frequency decomposition space associated with dyadic decomposition.
Surprisingly, as a frequency decomposition space associated with uniform decomposition,
the modulation space $M^{p,q}_s$ keeps the $M^{p,q}_s \rightarrow M^{p,q}_s $  boundedness of $\me^{\mi t\D}$ for all $1\leq p \leq \infty$.
This phenomenon was first discovered by B\'{e}nyi-Gr\"{o}chenig-Okoudjou-Rogers \cite{BenyiGroechenigOkoudjouRogers2007JFA},
and then it was developed and sharpen by Miyachi-Nicola-Rivetti-Tabacco-Tomita \cite{MiyachiNicolaRivettiTabaccoTomita2009PAMS}.
Thanks to the boundedness of  $\me^{\mi t\D}$, and its fractional form $\me^{\mi t(-\D)^{\beta/2}}$ between modulation spaces,
we can solve NLS and the fractional Schr\"{o}dinger equation with initial data belongs to modulation spaces $M^{p,q}_s$ for all $1\leq p \leq \infty$, see \cite{BenyiOkoudjou2009BLMS, GuoChen2014FMC}.

Modulation spaces was introduced firstly by Feichtinger \cite{Feichtinger1983} in 1983 to
give a simultaneous description of temporal and frequency behavior for a function or distribution. %
The study of modulation has over time been transformed into a rich and multifaceted theory,
providing basic insights into such topics as harmonic analysis, time-frequency analysis
and partial differential equations. Nowadays, the theory has played more and
more notable roles.
One can refer \cite{GalperinSamarah2004ACHA, RuzhanskySugimotoToftTomita2011MN} for some basic properties of modulation spaces, and \cite{CorderoNicola2009BAMS, GuoChenFanZhao2016ataiMMJ, GuoFanWuZhao2018SM} for the production and convolution properties on (weighted)  modulation spaces,
and \cite{Tomita2006MN, SugimotoTomita2008MN, ZhongChen2011SCM}  for the boundedness of fractional integrals,
and also \cite{ZhaoFanGuo2018AFA} for the boundedness of Hausdorff operators on modulation spaces, and see
\cite{BenyiGrafakosGroechenigOkoudjou2005ACHA, BenyiGroechenigOkoudjouRogers2007JFA, WangHudzik2007JDE, BenyiOkoudjou2009BLMS, CorderoNicola2009JMAA}
for the study of nonlinear evolution equations on modulation spaces.
For the boundedness of unimodular multipliers $\me^{\mi\mu(D)}$
between modulation spaces, one can also see some recent articles \cite{ChenFanSun2012DCDS, ZhaoChenFanGuo2016NA, NicolaPrimoTabacco2018JoPOaA}.

As mentioned before, modulation and Besov spaces
are frequency decomposition spaces associated with uniform and dyadic decomposition respectively.
As an intermediate decomposition between the dyadic and uniform decomposition,
the $\a$-covering was first introduced by Ferchtinger \cite{FeichtingerGroebner1985MN, Feichtinger1987MN}.
Then, using the $\alpha$-covering
of the frequency plane, Gr\"{o}bner \cite{Grobner1992} introduced
the $\alpha $-modulation spaces $M_{p,q}^{s,\alpha }$ for $\alpha \in \lbrack 0,1)$.

Accordingly, the $\a$-modulation space (concrete definition in Section 2), generated by the $\a$-covering, is introduced formally as the intermediate spaces between modulation space and Besov space.
The space $M_{p,q}^{s,\alpha }$
coincides with the modulation space $M_{p,q}^{s}$ when $\alpha=0$, and
the (inhomogeneous) Besov space $\ B_{p,q}^{s}$ can be regarded as the limit case of
$M_{p,q}^{s,\alpha }$ as $\alpha\rightarrow 1 $ (see \cite{Grobner1992}). So, for the sake of convenience, we
can view the Besov space as a special $\alpha$-modulation space and use $%
M_{p,q}^{s,1}$ to denote the inhomogeneous Besov space $B_{p,q}^{s}$.
It is worth mentioning that $\a$-modulation spaces is NOT the interpolation space between modulation and Besov spaces \cite{GuoFanWuZhao2016JFAA}.
This fact reveals that the boundedness result on $\alpha$-modulation spaces can not be automatically valid by the corresponding results on
modulation and Besov spaces.

Among numerous references on $\a$-modulation spaces,
one can see \cite{HanWang2014JMSJ, ToftWahlberg2012PSMB} for elementary properties of $\a$-modulation spaces,
see \cite{GuoFanZhao2018SCM} for the full characterization of embedding between $\a$-modulation spaces,
see \cite{WuChen2014AMJCUSB, ZhaoFanGuo2016MN,ZhaoGuoYu2018BAMS} for the research of boundedness of fractional integrals
and see \cite{Borup2004JFSA,BorupNielsen2006JMAA,BorupNielsen2006AM,BorupNielsen2006MN} for the study of pseudodifferential operators and nonlinear appoximation.
We also point out that the boundedness of unimodular multipliers on $\a$-modulation spaces has been studied
in \cite{ZhaoChenGuo2014JFS, ZhaoChenFanGuo2016NA}, in which if we take $\a=0$, the result is accordance to that in \cite{MiyachiNicolaRivettiTabaccoTomita2009PAMS}. Denote by $[t]$ the integer part of $t\in\mathbb{R}$.
The main boundedness result on $\alpha$-modulation space of unimodular multipliers can be stated as follows.
~

\noindent \textbf{Theorem A (\cite{ZhaoChenGuo2014JFS,ZhaoChenFanGuo2016NA})}
Let $1\leq p,q \leq \infty, s\in \mathbb{R}, \alpha \in [0,1]$.
Assume that $\mu$ is a real-valued function of class $C^{[n/2]+3}(\mathbb{R}^n \backslash \{0\})$  which satisfies
\begin{equation*}
|\partial^{\r}(\langle \xi\rangle^{2\al-2s} \partial^{\g}\mu(\xi))|\leq C_{\gamma}\langle\xi\rangle^{-|\r|}, \hspace{10mm} ~|\r|\leq[n/2]+1
\end{equation*}
for all $|\g|=2$.
Then we have
\begin{equation*}
\|e^{i\mu(D)}f\|_{M^{p,q}_{0,\alpha}} \leq C \|f\|_{M^{p,q}_{\d n,\alpha}},
\end{equation*}
with $\d \geq |1/p-1/2|\max\{2s,0\}$, where the constant $C$ is independent of $f$.

~

In order to compare with the results of this paper, Theorem A is stated by an equivalent form of the corresponding boundedness results in \cite{ZhaoChenGuo2014JFS,ZhaoChenFanGuo2016NA}, where the potential loss has been proved to be sharp.

Recently, in the case of modulation space, Nicola-Primo-Tabacoo \cite{NicolaPrimoTabacco2018JoPOaA} use a more soft and elegant method to deal with the boundedness of unimodular multipliers.
Inspired by this, we further consider the boundedness of unimodular multipliers on $\a$-modulation spaces.

More precisely, we establish the boundedness on $\a$-modulation spaces of $e^{i\mu(D)}$
by a new method derived from \cite{NicolaPrimoTabacco2018JoPOaA}.
In contrast to the previous results as in \cite{ZhaoChenGuo2014JFS, ZhaoChenFanGuo2016NA},
our results is valid under a weaker condition, see Remark \ref{TH.Vs.Known} for more details.
Since the $\alpha$-covering for $\alpha\in (0,1]$ is not uniform bounded as the case of modulation space ($\alpha=0$),
we refine the technique in \cite{NicolaPrimoTabacco2018JoPOaA}, making it more efficient and adaptable to our situation.

Denote by $\|f\|_{\scrF M^{1,\i}_{s,\a}}:=\|\scrF^{-1} f\|_{ M^{1,\i}_{s,\a}}$.
Our main result is stated as follows.


\begin{theorem}[Boundedness of unimodular multiplier on $\alpha$-modulation space]\label{TH.Bdd.unim.multipl.AlphaModu.}
  Let $1\leq p,q\leq\infty$, $s\in\bbR$, $\al\in[0,1]$ and $\mu\in C^2(\bbR^n)$ be a real valued function satisfying
  \begin{equation}
    \partial^{\g}\mu\in \scrF M^{1,\i}_{2\al-2s,\a} \text{ \, for all multi-index } \g \text{ with } |\g|=2.
  \end{equation}
  Then
  \begin{equation}
    \me^{\mi\mu(D)}:\mpqda\rightarrow\mpqa
  \end{equation}
  is bounded with
  $\d\geq  \left|1/p-1/2\right|  \max\left\{2s,0\right\}$.
\end{theorem}


\begin{remark}\label{TH.Vs.Known}
We would like to point out that our new boundedness result has a wider application range, since
our new assumption $\partial^{\g}\mu\in \scrF M^{1,\i}_{2\al-2s,\a}$ is weaker than the assumptions on $\mu$ in the previous results.
In fact, denote by
\be
N_{2\al-2s}=\{f: |\partial^{\r}(\langle \xi\rangle^{2\al-2s} f)|\leq C_{\gamma}\langle\xi\rangle^{-|\r|}, |\r|\leq[n/2]+1\}.
\ee
The assumption of Theorem A can be stated as follows:
  \begin{equation}
    \partial^{\g}\mu\in N_{2\al-2s} \text{ \, for all multi-index } \g \text{ with } |\g|=2.
  \end{equation}
And, we can verify that $N_{2\al-2s} \subsetneqq \scrF M^{1,\i}_{2\al-2s,\a}$. The proof will be presented at the end of Section 3.
\end{remark}

For simplicity, we only give the proof of Theorem \ref{TH.Bdd.unim.multipl.AlphaModu.} for the cases $\a\in[0,1)$, the proof of $\a=1$ is similar.
The paper is organized as follows.
In Section 2 we list some definitions, lemmas and give some key propositions which will be used lately.
The proof of main theorem will be given in Section 3.
We also give some details for Remark \ref{TH.Vs.Known}.

\section{Definitions and Lemmas}
The notation $X\lesssim  Y$ denotes the statement that $X\leq CY$,
the notation $X\sim Y$ means the statement $X\lesssim Y \lesssim X$.
$L^p$ denotes the usual Lebesgue spaces with $1 \leq p \leq \infty$, and we denote its norm by $\|\cdot\|_{L^p}$.
Let $\mathscr {S}:= \mathscr {S}(\mathbb{R}^{n})$ be the Schwartz space
and $\mathscr {S}':=\mathscr {S}'(\mathbb{R}^{n})$ be the space of tempered distributions.
For $x=(x_1,x_2,...,x_n)\in \mathbb{R}^{n}$,
we denote
\begin{equation*}
|x|:=\bigg( \sum_{j=1}^{n}|x_{j}|^2 \bigg)^{1/2}
\text{ \ \ \ and \ \ }
\langle x\rangle :=\left( 1+|x|^{2} \right)^{1/2}.
\end{equation*}
For a multi-index $\a=(\a_1,\a_2,\cdots,\a_n)$ with $\a_j\in\bbN =\bbZ^+\cup\{0\}$ for $j=1,2,\cdots,n$,
we denote
\[
|\a|:=\sum_{j=1}^n \a_j
\text{ \ \ \ and \ \ }
\a! :=\a_1!\cdot\a_2!\cdot\cdots\cdot\a_n!,
\]
with conventional rules $0!=1$.
The Fourier transform $\scrF f$ and the inverse Fourier transform $\scrF^{-1}f$ of $f\in \mathscr {S}(\mathbb{R}^{n})$ are defined by
$$
\scrF f(\xi)=\hat{f}(\xi)=\int_{\mathbb{R}^{n}}f(x)e^{-2\pi ix\cdot \xi}dx
~~
,
~~
\scrF^{-1}f(x)=\hat{f}(-x)=\int_{\mathbb{R}^{n}}f(\xi)e^{2\pi ix\cdot \xi}d\xi.
$$
We denote $\|f\|_{\scrF L^p}:=\|\scrF^{-1}f\|_{L^p}$.

Next we recall the definition of $\a$-modulation spaces.
First we give the partition of unity  associated with $\alpha\in [0,1)$.
Take two appropriate constants $0<C_1<C_2$ and choose a Schwartz function sequence $\{\eta_k^{\alpha}\}_{k\in \mathbb{Z}^n}$
satisfying
\begin{equation}\label{decompositon function of alpha modulation spaces}
\begin{cases}
|\eta_k^{\alpha}(\xi)|= 1,   \text{~if~} |\xi-\kaak|\leq C_1 \kaa\};\\
\text{supp}\eta_k^{\a}\subset \{\xi\in\mathbb{R}^n: |\xi-\kaak| \leq C_2\kaa\};\\
\sum\limits_{k\in \bbZ^{n}}\eta_k^{\alpha}(\xi)\equiv 1, \forall \xi\in \mathbb{R}^{n};\\
|\partial^{\gamma}\eta_k^{\alpha}(\xi)|\leq C_{|\alpha|}\lkr^{-\frac{\alpha|\gamma|}{1-\alpha}} , \forall \xi\in \mathbb{R}^{n}, \gamma \in\bbN^{n}.
\end{cases}
\end{equation}
The sequence $\{\eta_{k}^{\alpha}\}_{k\in\mathbb{Z}^{n}}$ constitutes a smooth partition of unity of $\mathbb{R}^{n}$.
The  frequency decomposition operators associated with the above function sequence are defined by
\begin{equation}
\Box_{k}^{\alpha}:= \scrF^{-1}\eta_{k}^{\alpha}\scrF, \, k\in \mathbb{Z}^{n}.
\end{equation}
Let $1\leq p,q \leq \infty$, $s\in \mathbb{R}$, $\alpha \in [0,1)$. The $\alpha$-modulation space associated with above decomposition is defined by
\begin{equation}
\mpqsa(\mathbb{R}^n)
=
\bigg\{
  f\in \mathscr {S}'(\mathbb{R}^{n}):
    \|f\|_{\mpqsa(\mathbb{R}^n)}
      =\bigg( \sum_{k\in \mathbb{Z}^{n}}\langle k\rangle ^{\frac{sq}{1-\alpha}}\|\Box_k^{\alpha} f\|_{L^p}^{q}\bigg)^{1/q}<\infty
\bigg\},
\end{equation}
with the usual modifications when $q=\infty$.
The modulation spaces coincides with $\a$-modulation spaces when $\a=0$.
And we have
$\|f\|_{\scrF \mpqsa}
=\|\scrF^{-1} f\|_{\mpqsa}
=\Big( \sum\limits_{k\in \mathbb{Z}^{n}}\langle k\rangle ^{\frac{sq}{1-\alpha}}\|\eta_k^{\a}f\|_{\scrF L^p}^{q}\Big)^{1/q}
$,
with the usual modifications when $q=\infty$.

\begin{remark}
The definition of $\a$-modulation spaces is independent of the choice of exact $\eta_k^{\alpha}$ (see \cite{HanWang2014JMSJ}).
\end{remark}
To define the Besov spaces, we introduce  the dyadic decomposition of $\mathbb{R}^n$.
Let $\varphi: \bbR^n\rightarrow [0,1]$ be a smooth bump function supported in the ball $\{\xi: |\xi|\leq3/2\}$ and equal 1 on the ball $\{\xi: |\xi|\leq 4/3\}$.
Denote
\begin{equation}\label{decompositon function of Besov space}
\phi(\xi):=\varphi(\xi)-\varphi(2\xi),
\end{equation}
and a function sequence
\begin{equation}
\begin{cases}
\phi_j(\xi)=\phi(2^{-j}\xi),~j\in \mathbb{Z}^{+},
\\
\phi_0(\xi)=1-\sum\limits_{j\in \mathbb{Z}^+}\phi_j(\xi)=\varphi(\xi).
\end{cases}
\end{equation}
For $j\in \bbN$, we define the Littlewood-Paley operators
\begin{equation}
\Delta_j:=\scrF^{-1}\phi_j\scrF.
\end{equation}
Let $1\leq p,q\leq\infty$, $s\in \mathbb{R}$. For a tempered distribution  $\ f$,  we set the norm
\begin{equation}
\|f\|_{B^{p,q}_s}:=\left(\sum_{j=0}^{\infty}2^{jsq}\|\Delta_jf\|_{L^p}^q \right)^{1/q},
\end{equation}
with the usual modifications when $q=\infty$.
The (inhomogeneous) Besov  space $B^{p,q}_s$ is the space of all tempered distributions $f$ for which the quantity $\|f\|_{B^{p,q}_s}$ is finite.


We now list some basic properties about $\alpha$-modulation spaces and Besov spaces.
As mentioned before, Besov space can be regarded as the limit case of $\a$-modulation space as $\a\rightarrow1$, so we also use $M^{p,q}_{s,1}$ to denote the (inhomogeneous) Besov space $B^{p,q}_s$.

%
\begin{lemma}[Bernstein multiplier theorem, see \cite{ZhaoFanGuo2016MN}]\label{L.BernsteinIneq.}
  Let $0<p\leq 2$. Assume $\partial^\gamma f\in L^2$ for all multi-indices $ \gamma $ with $|\gamma|\leq [n(1/p-1/2)]+1 $,
  where $[t]$ denotes the integer part of $t\in\mathbb{R}$. We have
\begin{equation}
\|\scrF^{-1}f\|_{L^p} \lesssim \sum_{\scriptscriptstyle{|\gamma|\leq \left[n(1/p \! - \! 1/2)\right]+1}} \|\partial^\gamma f\|_{L^2}.
\end{equation}

\end{lemma}
Using Bernstein multiplier theorem, we give some $\scrF L^1$-norm estimates of the decomposition function $\eta_k^\a$.
We adopt the following notation for convenience:
\[\Upsilon_k^{\a}(f)(\xi):=f(\kaa\xi+\kaak).\]
\begin{proposition}[Estimate of the decomposition function]\label{Prop.eta-k-a in FL1}
  Let $\{\eta_k^\a\}_{k\in\bbZ^n}$ be a smooth decomposition of $\bbR^n$ satisfying (\ref{decompositon function of alpha modulation spaces}).
  Then there exist constants $c$ and $C$, such that for all $k\in\bbZ^n$,
  \begin{enumerate}
    \item uniform support:
          \[\text{supp } \Upsilon_k^{\a}(\eta_k^\a)=\text{supp }\eta_k^\a(\kaa\cdot+\kaak) \subset B(0,c);\]
    \item uniform $\scrF L^1$ bound: 
            \[\|\Upsilon_k^{\a}(\eta_k^\a)\|_{\scrF L^1}=\|\eta_k^\a\|_{\scrF L^1}\leq C.\]
  \end{enumerate}

\end{proposition}
\begin{proof}
  The first conclusion can be derived directly
  by the definition of $\{\eta_k^{\a}\}_{k\in\bbZ^n}$ in (\ref{decompositon function of alpha modulation spaces}).
  We turn to prove the second one.
 Using Lemma \ref{L.BernsteinIneq.}, we have
  \begin{align*}
    \|\Upsilon_k^{\a}(\eta_k^\a)\|_{\scrF L^1}
    &=\|\eta_k^\a\|_{\scrF L^1}
        = \|\eta_k^\a(\lkr^{\aa}\cdot)\|_{\scrF L^1}
    \\
    &
    \lesssim
    \sum_{\scriptscriptstyle{|\gamma|\leq \left[n(1/p \! - \! 1/2)\right]+1}}
      \left\| \partial^\gamma \left(\eta_k^\a(\lkr^{\aa}\cdot)\right) \right\|_{L^2}
    \\
    &
    \lesssim
    \sum_{\scriptscriptstyle{|\gamma|\leq \left[n(1/p \! - \! 1/2)\right]+1}}
      \lkr^{\aa(|\gamma|-n/2)}
      \cdot
      \left\| \partial^\gamma \eta_k^\a \right\|_{L^{\i}}\lkr^{\frac{n\al/2}{1-\a}}
    \lesssim1,
  \end{align*}
  where the last inequality we use the derivative property and support information of $\eta_k^{\a}$ as mentioned in (\ref{decompositon function of alpha modulation spaces}).
\end{proof}
\begin{proposition}[Convolution of $\alpha$-modulation space]\label{TH.conv.AlphaModu.}
  Let $1\leq p, q \leq \infty$, $s\in\bbR$, $\alpha\in[0,1]$.
  Then we have
  \begin{equation}
    \mpqsa\ast M^{1,\infty}_{-s,\al}\subset\mpqa.
  \end{equation}
\end{proposition}
\begin{proof}
  For $k\in\bbZ^n$, denote  $\G_k:=\{ \l\in\bbZ^n |\, \bka\Box_{\l}^{\a}\neq 0\}$,
  $\wt{\eta}_k^{\a}:=\sum\limits_{\l\in\G_k}\eta_{\l}^\a$
  and
  \[
    \bkat=\sum_{\l\in\G_k} \Box_{\l}^\a.
  \]
  We have  $\text{supp}\eta_k^{\a}\subset\text{supp}\wt{\eta}_k^{\a}\subset B(\kaak, C\kaa)$ with some fixed constant $C$ for all $k\in\bbZ^n$.
  and
  $\bka(f\ast g)
  =\scrF^{-1}(\eta_k^{\a}\cdot\wh{f}\cdot\wh{g})
  =\scrF^{-1}(\eta_k^{\a}\cdot\wh{f}\cdot\wt{\eta}_k^{\a} \cdot \wh{g})
  =\bka f \ast \bkat g
  $.
  Using Young inequality
  we have
  \begin{align*}
    \|f\ast g\|_{\mpqa}
    &
    =\|\{\|\bka(f\ast g)\|_{\lp}\}_{k\in\bbZ^n}\|_{l^q}
     =\|\{\|\bka f\ast \bkat g\|_{\lp}\}_{k\in\bbZ^n}\|_{l^q}
    \\
    &
    \lesssim
     \|
       \{ 
            \|\bka f\|_{\lp}
          \cdot  \|\bkat g\|_{L^1}
       \}_{k\in\bbZ^n}
     \|_{l^q}
    \\
    &\leq
     \|\{  \lkr^{\frac{s}{1-\a}} \|\bka f\|_{\lp}  \}_{k\in\bbZ^n} \|_{l^q}
     \cdot
     \|\{  
           \lkr^{-\frac{s}{1-\a}} \|\bkat g\|_{L^1}
       \}_{k\in\bbZ^n}
     \|_{l^{\infty}}
    \\
    &\sim
    \|f\|_{\mpqsa} \|g\|_{M^{1,\infty}_{-s,\a}}.
  \end{align*}
\end{proof}

\section{Boundedness of $\me^{\mi\mu(D)}$ on $\a$-modulation space}

%
This section is devoted to the boundedness of unimodular multipliers on $\alpha$-modulation space.
We first establish a bounded results, which is sharp at two endpoints  $p=1$ and $p=\i$.
Then the desired conclusion follows by an interpolation between this and the obvious boundedness of $M^{2,q}_{0,\al}$.
\begin{theorem}[Boundedness for endpoints]\label{TH.Bdd.unim.multipl.AlphaModu.-Special}
  Let $1\leq p,q\leq\infty$, $s\in\bbR$, $\al\in[0,1]$ and $\mu\in C^2(\bbR^n)$ be a real-valued function satisfying
  \begin{equation}\label{condi.mu-Special}
    \partial^{\g}\mu\in \scrF M^{1,\i}_{2\a-2s,\a}, \text{ \, for all multi-index } \g \text{ with } |\g|=2.
  \end{equation}
  Then
  \begin{equation}
    \me^{\mi\mu(D)}:\mpqda\rightarrow\mpqa
  \end{equation}
  is bounded with $\d\geq\max\{s, 0\}$.
\end{theorem}
\begin{proof}
It is sufficient to give the proof for the case $\d=\max\{s, 0\}$,
since the other cases can be obtained by the simple embedding relation $M^{p,q}_{t_1,\a}\subset M^{p,q}_{t_2,\a}$ ($t_2\leq t_1$).
Write $\me^{\mi\mu(D)} f  =  \scrF^{-1}\me^{\mi\mu} \ast f$.
By Proposition \ref{TH.conv.AlphaModu.}, the desired conclusion is valid if we prove
$\|\scrF^{-1}\me^{\mi\mu}\|_{M^{1,\infty}_{-\d n,\a}}<\infty$, i.e.
\begin{equation}\label{TH.Bdd.unim.multipl.AlphaModu.-Special-for proof 1}
  \sup_{k\in\bbZ^n} \lkr^{-\f{\d n}{1-\a}} \|\eta_k^\a \cdot \me^{\mi\mu(\cdot)} \|_{\scrF L^1}<\infty.
\end{equation}
Denote
\[
E_k:=
   \left\{
     \l\in\bbZ^n\bigg|
       \eta_\l^\a\left(\cdot\right)
       \cdot
       \eta_k^\a\left(\lkr^{\f{-\d}{(1-\a)^2}}\cdot\right)
     \neq0
   \right\}.
\]
Recall that $\{\eta_k^\a\}_{k\in\bbZ^n}$ is a partition of unity on $\bbR^n$, we have
\[
\eta_k^\a\Big(\lkr^{\f{-\d}{(1-\a)^2}}\xi\Big)
=
\sum_{\l\in E_k}
 \eta_\l^\a(\xi)
 \eta_k^\a\Big(\lkr^{\f{-\d}{(1-\a)^2}}\xi\Big).
\]
Then we have the the following assertion:
\begin{equation}\label{Pf.numb.Ek-Special}
  \llr\sim\lkr^{1+\d/(1-\a)} \text{ for all }\l\in E_k,
  \text{ and }
  \#E_k\sim\lkr^{\d n/(1-\a)}
\end{equation}
where $\#E_k$ is the cardinality of $E_k$, and that
\begin{equation}\label{Pf.distance.Ek-Special}
  | \kaak  -  \lkr^{\f{-\d}{(1-\a)^2}}\laal |
  \lesssim
  \kaa
  \text{ for all }\l\in E_k.
\end{equation}
In fact,
when $\l\in E_k$, the support sets of $\eta_\l^\a\left(\xi\right)$ and $\eta_k^\a\left(\lkr^{\f{-\d}{(1-\a)^2}}\xi\right)$  must be intersected,
then we have
$  \llr^{\frac{1}{1-\a}}
\sim
\ka \cdot \lkr^{\f{\d}{(1-\a)^2}} $,
i.e.
$\langle \l \rangle
\sim
\lkr^{1+\d/(1-\a)}$.
So $|\text{supp}\eta_{\l}^{\a}| \sim \big(\laa\big)^n \sim \Big(\lkr^{\aa+\f{\a\d}{(1-\a)^2}}\Big)^n$ for all $\l \in E_k$.
This and the almost orthogonality of $\{\eta_{\l}^{\a}\}_{\l\in E_k}$ yield that
\[
  \#E_k
  \sim
    \f{ \big|\text{supp} \eta_k^\a\big(\lkr^{\f{-\d}{(1-\a)^2}}\cdot\big)\big| }{ |\text{supp}\eta_{\l}^{\a}| }
  \sim
    \Bigg( \f{ \kaa\lkr^{\f{\d}{(1-\a)^2}} }{ \lkr^{\aa+\f{\a\d}{(1-\a)^2}} } \Bigg)^n
  \sim
    \lkr^{\d n/(1-\a)}.
\]
Furthermore, when $\l\in E_k$, using \eqref{Pf.numb.Ek-Special} and the position relation between $\eta_\l^\a$ and
 $\eta_k^\a\Big(\lkr^{\f{-\d}{(1-\a)^2}}\xi\Big)$,
we get
\[
| \lkr^{\f{\d}{(1-\a)^2}} \kaak  -  \laal |
\lesssim
\lkr^{\f{\d}{(1-\a)^2}} \kaa+\laa
\sim
 \kaa (\lkr^{\f{\d}{(1-\a)^2}}+\lkr^{\f{\a\d}{(1-\a)^2}}).
\]
This and the fact $\d\geq0$ imply that
\[
|  \kaak  - \lkr^{\f{-\d}{(1-\a)^2}} \laal |
\lesssim
 \kaa (1+\lkr^{\f{(\al-1)\d}{(1-\a)^2}})
\lesssim
\kaa.
\]
So we get (\ref{Pf.distance.Ek-Special}).

Denote by $\Upsilon_{k}^{\a}(f)(\xi):=f(\kaa\xi+\kaak)$ and
\[
\mu_k(\xi):=\mu\Big(\lkr^{\f{-\d}{(1-\a)^2}}\xi\Big),
\hspace{5mm}
\mu_{k,\l}(\xi):=\mu_k\Big(\llr^\aa\xi+ \laal\Big).
\]
Then the scaling invariance of $\scrF L^1$ and Young's inequality yield that
\begin{align}\label{Pf.multipl.0-Special}
  \|\eta_k^\a \cdot  \me^{\mi\mu(\xi)}\|_{\scrF L^1}
  &
  =
  \Big\| \eta_k^\a\Big(\lkr^{\f{-\d}{(1-\a)^2}}\xi\Big)
   \cdot  \me^{\mi\mu_k(\xi)}\Big\|_{\scrF L^1}
  =
  \bigg\|\sum_{\ell\in E_k}\eta_{\ell}^\a(\xi)
   \cdot \eta_k^\a\Big(\lkr^{\f{-\d}{(1-\a)^2}}\xi\Big)
   \cdot \me^{\mi\mu_k(\xi)}\bigg\|_{\scrF L^1}
   \nonumber
  \\
  &\leq
   \Big\| \eta_k^\a\Big(\lkr^{\f{-\d}{(1-\a)^2}}\xi\Big) \Big\|_{\scrF L^1}
    \cdot
    \bigg\| \sum_{\l\in E_k}\eta_{\l}^\a(\xi) \cdot \me^{\mi\mu_k(\xi)}\bigg\|_{\scrF L^1}
    \nonumber
  \\&\lesssim
   \sum_{\l\in E_k}
    \Big\| \eta_{\l}^\a(\xi)\cdot \me^{\mi\mu_k(\xi)}\Big\|_{\scrF L^1}
  =
   \sum_{\l\in E_k}
     \left\|
      \Upsilon_{\l}^{\a} (\eta_{\l}^{\a})(\xi) \cdot \me^{\mi\mu_{k,\l}}
   \right\|_{\scrF L^1}.
\end{align}
Furthermore, we write
\begin{align}\label{Pf.multipl.1-Special}
  \left\|
      \Upsilon_{\l}^{\a} (\eta_{\l}^{\a})(\xi) \cdot \me^{\mi\mu_{k,\l}}
   \right\|_{\scrF L^1}
  =
   \left\|
      \Upsilon_{\l}^{\a} (\eta_{\l}^{\a})(\xi) \cdot \me^{\mi \psi^{\a}_{k,\l}(\xi)}
   \right\|_{\scrF L^1},
\end{align}
where
$
  \psi_{k,\l}^\a(\xi)
  :=
  \mu_{k,\l}(\xi)-\mu_{k,\l}(0)-\na\mu_{k,\l}(0)\cdot\xi
  =
  \sum\limits_{|\g|=2} \f{2}{\g!} \cdot \xi^{\g} \cdot \int_0^1 \partial^{\g}\mu_{k,\l}(\t\xi) \cdot (1-\t) \md \t.
$
Let $\eta^\ast\in C_c^{\infty}$ suppored in $B(0,2c)$ with $\eta^\ast(\xi)=1$ for $\xi\in B(0,c)$,
where $c$ is the radius of the uniform support of $\Upsilon_{\l}^{\a} (\eta_{\l}^{\a})(\xi)$ in Proposition \ref{Prop.eta-k-a in FL1}.
Observe that $\eta^\ast(\t \xi)\cdot \Upsilon_{\l}^{\a} (\eta_{\l}^{\a})(\xi)=\Upsilon_{\l}^{\a} (\eta_{\l}^{\a})(\xi)$ for $\t\in[0,1]$.
We have
\begin{align}
  \Upsilon_{\l}^{\a} (\eta_{\l}^{\a})(\xi) \cdot \me^{\mi \psi^{\a}_{k,\l}(\xi)}
   =
  \Upsilon_{\l}^{\a} (\eta_{\l}^{\a})(\xi) \cdot \me^{\mi \vph^{\a}_{k,\l}(\xi)}
  ,
  \text{ for all }\xi\in\bbR^n,
\end{align}
where
\[
  \vph_{k,\l}^\a(\xi)
  :=
  \sum\limits_{|\g|=2} \f{2}{\g!}
   \cdot \xi^{\g} \cdot \eta^\ast(\xi)
   \cdot \int_0^1 \partial^{\g}\mu_{k,\l}(\t\xi) \cdot \eta^\ast\left(\t\xi\right) \cdot (1-\t) \md \t.
\]
Furthermore, we claim that 
\begin{align}\label{Pf.multipl.2-Special}
\|\vph_{k,\l}^\a \|_{\scrF L^1}
\lesssim
\sum_{|\g|=2}
  \|\pa^{\g}\mu\|_{\scrF M^{1,\i}_{2\a-2\d,\a}}
,
\end{align}
for all $\l\in E_k$.
In fact,
since $\xi^{\g}\eta^{\ast} \in C_c^{\infty} \subset \scrF L^1$ for all $|\g|=2$,
a direct calculation shows that
\begin{align*}
\|\vph_{k,\l}^\a \|_{\scrF L^1}
&\leq
 \sum_{|\g|=2} \f{2}{\g!}
    \left\| \xi^{\g} \cdot \eta^\ast(\xi) \right\|_{\scrF L^1}
     \cdot
     \left\| \int_0^1 \partial^{\g}\mu_{k,\l}(\t\xi) \cdot \eta^\ast\left(\t\xi\right) \cdot (1-\t) \md \t \right\|_{\scrF L^1}
\\
&
\lesssim
\sum_{|\g|=2}
     \left\| \partial^{\g}\mu_{k,\l}(\xi) \cdot\eta^\ast(\xi) \right\|_{\scrF L^1}
\\
&
=
\sum_{|\g|=2} \llr^\ata
      \left\| \partial^{\g}\mu_k\big(\laa\xi+\laal\big)
            \cdot \eta^\ast(\xi)
      \right\|_{\scrF L^1}
\\
&
=
\sum_{|\g|=2} \llr^\ata
      \left\| \partial^{\g}\mu_k(\xi)
            \cdot \eta^\ast\left(\llr^\naa\xi-\l\right)
      \right\|_{\scrF L^1}
\\
&
=
\sum_{|\g|=2} \llr^\ata
      \lkr^{\f{-2\d}{(1-\a)^2}}
      \left\| \partial^{\g}\mu\left(\lkr^{\f{-\d}{(1-\a)^2}} \xi \right)
              \cdot \eta^\ast\left(\llr^\naa\xi-\l\right)
      \right\|_{\scrF L^1}.
\end{align*}
By (\ref{Pf.numb.Ek-Special}), we further have
\begin{align*}
\|\vph_{k,\l}^\a \|_{\scrF L^1}
&
\lesssim
\sum_{|\g|=2}
      \lkr^{\f{2\a}{1-\a}\left(1+\f{\d}{1-\a}\right)+\f{-2\d}{(1-\a)^2}}
      \left\| \partial^{\g}\mu\left( \xi \right)
              \cdot \eta^\ast\left(\llr^\naa\lkr^{\f{\d}{(1-\a)^2}}\xi-\l\right)
      \right\|_{\scrF L^1}
\\
&
\sim
\sum_{|\g|=2}
      \lkr^{\f{2\a-2\d}{1-\a}}
      \left\| \partial^{\g}\mu\left( \xi \right)
              \cdot \eta^\ast\Big(\llr^\naa\lkr^{\f{\d}{(1-\a)^2}}\xi-\l\Big)
      \right\|_{\scrF L^1}
.
\end{align*}
Recalling that $\text{supp}\eta^\ast \subset B(0,2c)$ and $\l\in E_k$,
$\forall \xi\in \text{supp}\eta^\ast\Big(\llr^\naa\lkr^{\f{\d}{(1-\a)^2}}\cdot-\l\Big)$,
we have
\[|\xi-\lkr^{\f{-\d}{(1-\a)^2}}\laal|\leq 2c \lkr^{\f{-\d}{(1-\a)^2}}\laa.\]
Combining this with \eqref{Pf.numb.Ek-Special} and (\ref{Pf.distance.Ek-Special}), we have
\begin{align*}
  |\xi-\kaak|
  &
  \leq
  |\xi-\lkr^{\f{-\d}{(1-\a)^2}}\laal|  +  |\lkr^{\f{-\d}{(1-\a)^2}}\laal-\kaak|
  \\
  &
  \lesssim
  \lkr^{\f{-\d}{(1-\a)^2}}\laa  +  \kaa
  \lesssim
  \kaa,
\end{align*}
for all $ \xi\in \text{supp}\eta^\ast\Big(\llr^\naa\lkr^{\f{\d}{(1-\a)^2}}\cdot-\l\Big)$ and $ \l\in E_k$.
Hence, there exist a constant $C_3$ such that
\[
 \text{supp}\eta^\ast\Big(\llr^\naa\lkr^{\f{\d}{(1-\a)^2}}\cdot-\l\Big)
 \subset
 B(\kaak,C_3\kaa),
  \forall \l\in E_k.
\]
Then we have $\#F_{k,\l}\lesssim 1$, and that $\langle k_0 \rangle  \sim  \langle k \rangle$ for all $k_0\in F_{k,\l}$ and $ \l\in E_k$,
where we denote
\[
 F_{k,\l}:=
 \Big\{
  k_0\in\bbZ^n,
  \eta_{k_0}^{\a}(\cdot) \cdot \eta^\ast\Big(\llr^\naa\lkr^{\f{\d}{(1-\a)^2}}\cdot-\l\Big)  \neq0
 \Big\}.
\]
Then by the translation and scaling invariance of $\scrF L^1$ and the definition of $\a$-modulation spaces, we further deduce that
\begin{align*}
\|\vph_{k,\l}^\a \|_{\scrF L^1}
&\lesssim
\sum_{|\g|=2}
      \lkr^{\f{2\a-2\d}{1-\a}}
      \bigg\|
       \sum_{k_0\in F_{k,\l}}
       \partial^{\g}\mu(\xi)
        \cdot
        \eta^\ast\left(\llr^\naa\lkr^{\f{\d}{(1-\a)^2}}\xi-\l\right)
        \cdot
       \eta_{k_0}^{\a}(\xi)
      \bigg\|_{\scrF L^1}
\\
&
\lesssim
\sum_{|\g|=2}
      \lkr^{\f{2\a-2\d}{1-\a}}
      \sum_{k_0\in F_{k,\l}}
      \left\|
       \partial^{\g}\mu(\xi)
        \cdot
       \eta_{k_0}^{\a}(\xi)
      \right\|_{\scrF L^1}
\\
&
\sim
\sum_{|\g|=2}
      \sum_{k_0\in F_{k,\l}}
      \langle k_0 \rangle^{\f{2\a-2\d}{1-\a}}
      \left\|
       \partial^{\g}\mu(\xi)
        \cdot
       \eta_{k_0}^{\a}(\xi)
      \right\|_{\scrF L^1}
\lesssim
\sum_{|\g|=2}
  \|\pa^{\g}\mu\|_{\scrF M^{1,\i}_{2\a-2\d,\a}}
,
\end{align*}
for all $ \l\in E_k$.
Therefore, we get (\ref{Pf.multipl.2-Special}).

Combine this with (\ref{Pf.multipl.0-Special}), (\ref{Pf.multipl.1-Special}) and Proposition \ref{Prop.eta-k-a in FL1}, we have
\begin{equation*}
\begin{split}
  \|\eta_k^\a \cdot  \me^{\mi\mu(\xi)}\|_{\scrF L^1}
  &
  \leq
  \sum_{\l\in E_k}
   \left(
    \left\|\Upsilon_{\l}^{\a} (\eta_{\l}^{\a})(\xi)\right\|_{\scrF L^1}
    +
    \left\|
     \Upsilon_{\l}^{\a} (\eta_{\l}^{\a})(\xi)
     \cdot
     \left( \me^{\mi \vph^{\a}_{k,\l}(\xi)}-1 \right)
    \right\|_{\scrF L^1}
   \right)
  \\
  &
  \lesssim
    \sum_{l\in E_k}
     \left( 1+ \left\|\me^{\mi\vph_{k,\l}^\a(\xi)}-1\right\|_{\scrF L^1}
     \right)
   =
    \sum_{l\in E_k}
     \Bigg( 1+ \bigg\|\sum_{m=1}^{\i} \f{1} {m!} \left(\mi\vph_{k,\l}^\a(\xi)\right)^m  \bigg\|_{\scrF L^1}
     \Bigg)
  \\
  &
   \leq
    \sum_{\l\in E_k}
    \bigg(
     1+
    \sum_{m=1}^{\i} \f{1} {m!} \left\|\vph_{k,\l}^\a\right\|_{\scrF L^1}^m
    \bigg)
   \sim
    \sum_{\l\in E_k} \exp(\left\|\vph_{k,\l}^\a\right\|_{\scrF L^1})
  \\
  &
   \lesssim
   \sum_{l\in E_k}
     \exp \bigg( C\sum\limits_{|\g|=2}\|\pa^{\g}\mu\|_{\scrF M^{1,\i}_{2\a-2\d,\a}} \bigg)
     =
     \#E_k  \cdot  \exp \bigg( C\sum\limits_{|\g|=2}\|\pa^{\g}\mu\|_{\scrF M^{1,\i}_{2\a-2\d,\a}} \bigg).
\end{split}
\end{equation*}
Drawing support from (\ref{condi.mu-Special}) and (\ref{Pf.numb.Ek-Special}), we obtain
\begin{align*}
  \|\eta_k^\a  \me^{\mi\mu(\xi)}\|_{\scrF L^1}
  \lesssim\lkr^{\frac{\d n}{1-\a}}.
\end{align*}
Then we get (\ref{TH.Bdd.unim.multipl.AlphaModu.-Special-for proof 1}) and complete the proof.
\end{proof}

\textbf{Proof of Theorem \ref{TH.Bdd.unim.multipl.AlphaModu.}.}
Note that $\mu\in C^2$ and $\partial^\g \mu \in \scrF M^{1,\infty}_{2\al-2s,\a}$ for $|\g|=2$.
By Theorem \ref{TH.Bdd.unim.multipl.AlphaModu.-Special}, taking $p=1$ and $p=\infty$, we have
\ben\label{TH.Bdd.unim.multipl.AlphaModu.-for proof 1}
  \left\| \me^{\mi \mu(D)}f \right\|_{M^{1,q}_{0,\a}}
  \lesssim
  \left\| f\right \|_{M^{1,q}_{\max\{s,0\}n/2,\a}}
\text{ and }
  \left\| \me^{\mi \mu(D)}f \right\|_{M^{\i,q}_{0,\a}}
  \lesssim
  \left\| f\right \|_{M^{\i,q}_{\max\{s,0\}n/2,\a}}.
\een
On the other hand, by the Plancherel equality we have
\begin{equation}\label{TH.Bdd.unim.multipl.AlphaModu.-Plancherel}
 \begin{split}
  \left\| \me^{\mi \mu(D)}f \right\|_{M^{2,q}_{0,\a}}
  &
  =
  \| \{\|\bka \me^{\mi \mu(D)}f\|_{L^2}\}_{\l\in\bbZ^n} \|_{\l^q}
  \sim
  \| \{\|\eta_k^\a \me^{\mi \mu(\xi)}\scrF f(\xi)\|_{L^2}\}_{\l\in\bbZ^n} \|_{\l^q}
  \\
  &
  \sim
  \| \{\|\eta_k^\a \scrF f(\xi)\|_{L^2}\}_{\l\in\bbZ^n} \|_{\l^q}
  \sim
  \| \{\|\bka f\|_{L^2}\}_{\l\in\bbZ^n} \|_{\l^q}
  \sim
  \left\| f\right \|_{M^{2,q}_{0,\a}}.
\end{split}
\end{equation}

An interpolation argument then yields the desired conclusion.
To be more specific,
 for $1\leq p<2$, applying complex interpolation theory between
 (\ref{TH.Bdd.unim.multipl.AlphaModu.-Plancherel}) and the first inequality in (\ref{TH.Bdd.unim.multipl.AlphaModu.-for proof 1}) with $\th=2(1/p-1/2)$ and $1/{p_{\th}}=(1-\th)/2+\th/1=(1+\th)/2=1/p$,
 we have
 \[
   \left\| \me^{\mi \mu(D)}f \right\|_{M^{p,q}_{0,\a}}
     \lesssim
   \left\| f\right \|_{M^{p,q}_{n\left(1/p-1/2\right) \max\{2s,0\},\a}}.
 \]
 While for  for $2\leq p \leq \i$, applying complex interpolation theory between
 (\ref{TH.Bdd.unim.multipl.AlphaModu.-Plancherel}) and the second inequality in (\ref{TH.Bdd.unim.multipl.AlphaModu.-for proof 1}) with $\th=2(1/2-1/p)$ and $1/{p_{\th}}=(1-\th)/2+\th/{\infty}=(1-\th)/2=1/p$,
we have
 \[
   \left\| \me^{\mi \mu(D)}f \right\|_{M^{p,q}_{0,\a}}
     \lesssim
   \left\| f\right \|_{M^{p,q}_{n\left(1/2-1/p\right) \max\{2s,0\},\a}}.
 \]
 Hence, we get the desired conclusion in Theorem $\ref{TH.Bdd.unim.multipl.AlphaModu.}$.

Now, we turn to give the proof for the relations $N_{2\al-2s} \subsetneqq \scrF M^{1,\i}_{2\al-2s,\a}$.
This will also indicate that our new result is an essential improvement of the previous results.

\textbf{Proof of $N_{2\al-2s} \subsetneqq \scrF M^{1,\i}_{2\al-2s,\a}$.}
The definition of these two function spaces imply that
$N_{2\al-2s} \subsetneqq \scrF M^{1,\i}_{2\al-2s,\a}\Longleftrightarrow N_{0} \subsetneqq \scrF M^{1,\i}_{0,\a}$.
We only verify the last one.
First, we verify the including relation 
$N_{0} \subset \scrF M^{1,\i}_{0,\a}$.
Take a function $g\in N_0$.
Using Lemma \ref{L.BernsteinIneq.}, we have
 \begin{align*}
   \|\eta_k^{\a}g\|_{\scrF L^1}
   &\lesssim
   \sum_{|\r|\leq[n/2]+1} \lkr^{\aa(|\r|-\f{n}{2})}
     \|\pa^{\r}(\eta_k^{\a}g)\|_{L^2}
   \\
   &\leq
   \sum_{|\r|\leq[n/2]+1} \lkr^{\aa(|\r|-\f{n}{2})}
     \sum_{\r_1+\r_2=\r} \| \pa^{\r_1}\eta_k^{\a}\cdot \pa^{\r_2}g \|_{L^2}
   \\
   &\lesssim
   \sum_{|\r|\leq[n/2]+1} \lkr^{\aa(|\r|-\f{n}{2})}
     \sum_{\r_1+\r_2=\r} 
        \left\| \pa^{\r_1}\eta_k^{\a}\right\|_{L^2}
        \cdot \lkr^{\f{-|\r_2|}{1-\a}}
   \\
   &\lesssim
   \sum_{|\r|\leq[n/2]+1} \lkr^{\aa(|\r|-\f{n}{2})}
     \sum_{\r_1+\r_2=\r} \lkr^{-|\r_1|\aa}
        \cdot \lkr^{\aa\f{n}{2}}
        \cdot \lkr^{\f{-|\r_2|}{1-\a}}\lesssim 1,
 \end{align*}
 where we use the fact that
 $| \pa^{\r_2} g |
 \leq  |\xi|^{-|\r_2|}
 \sim  \lkr^{-|\r_2|/(1-\a)}$
 for $\xi\in \text{supp}\eta_k^{\a}$.
 Thus,
 \begin{equation*}
 \|g\|_{\scrF M^{1,\i}_{0,\a}}
 =\sup_{k\in\bbZ^n}\|\eta_k^{\a}g\|_{\scrF L^1}
 \lesssim
 1,
 \end{equation*}
i.e. $g\in\scrF M^{1,\i}_{0,\a}$.
Therefore, $N_{0} \subset \scrF M^{1,\i}_{0,\a}$.

Next, we give an example shown that $N_0 \subsetneqq \scrF M^{1,\i}_{0,\a}$.
Take $h$ to be a smooth function with compact support near the origin such that its derivatives of all orders are not zero functions.
Set
\be
H(\xi)=\sum_{k\in \mathbb{Z}^n}h(\xi-\langle k\rangle^{\frac{\al}{1-\al}}k).
\ee
One can easily check that $H\in \scrF M^{1,\i}_{0,\a}$.
However, the derivatives of $H$ can not decay at infinity. So $H\notin N_0$.
\hfill $\qed$

\subsection*{Acknowledgements}
This work was partially supported by the National Natural Foundation of China (Nos. 11601456, 11701112, 11671414, 11771388)
and  Natural Science Foundation of Fujian Province (Nos. 2017J01723, 2018J01430).

\end{document}